\documentclass[11pt]{article}
\usepackage{amsthm}
\usepackage{amsmath}
\usepackage{amssymb}
\usepackage{amsfonts}
\usepackage[T1]{fontenc}
\newtheorem{theorem}{Theorem}

 \newtheorem{proposition}{Proposition}
 \newtheorem{corollary}{Corollary}
 
 \newtheorem{problem}{Problem}

\newcommand{\ie}{{i.e.}}
\newcommand{\eg}{{e.g.}}

\newcommand{\cent}{{\rm Center}}
\newcommand{\conv}{{\rm conv}}
\newcommand{\Delete}{{\rm Delete}}
\newcommand{\Order}{{\rm Order}}
\newcommand{\AnyOrder}{{\rm AnyOrder}}

\newcommand{\diam}{{\rm diam}}

\newcommand{\RR}{\mathbb{R}}  % set of real numbers
\newcommand{\Term}{Nearest Neighbor Graph}
\usepackage{graphicx}
\usepackage{hyperref}

\begin{document}
\title{Maximizing the Maximum Degree \\ in Ordered \Term s\thanks{An earlier version of this article
    appeared in Proceedings of CALDAM 2025.}} 

\author{P\'eter \'Agoston\thanks{Alfréd Rényi Institute of Mathematics, Budapest, Hungary\@.
    Email: \texttt{agostonp95@gmail.com}} \and
  \!\!\!\!Adrian Dumitrescu\thanks{Algoresearch L.L.C., Milwaukee, WI, USA,
  E-mail: \texttt{ad.dumitrescu@algoresearch.org}} \and
  \!\!\!\!Arsenii Sagdeev\thanks{ Karlsruhe Institute of Technology, Karlsruhe, Germany\@.
 Email: \texttt{sagdeevarsenii@gmail.com}} \and
  \!\!\!\!Karamjeet Singh\thanks{Indraprastha Institute of Information Technology, Delhi, India. 
Email: \texttt{karamjeets@iiitd.ac.in}} \and
  \!\!\!\!Ji Zeng\thanks{University of California San Diego, La Jolla, CA, USA and
    Alfréd Rényi Institute of Mathematics, Budapest, Hungary. Email: \texttt{jzeng@ucsd.edu}}
}

\maketitle

\begin{abstract}
For an ordered point set in a Euclidean space or, more generally, in an abstract metric space, the
\textit{ordered \Term} is obtained by connecting each  
of the points to its closest predecessor by a directed edge. We show that for every set of $n$ 
points in $\RR^{d}$, there exists an order such that 
the corresponding ordered \Term\ has maximum degree at least 
 $\log{n}/(4d)$. Apart from the $1/(4d)$ factor, this bound is the best possible. As for the
abstract setting, we show that for every $n$-element metric space, there exists an order such that
the corresponding ordered \Term\ has maximum degree $\Omega(\sqrt{\log{n}/\log\log{n}})$. 
\end{abstract}

\section{Introduction} \label{sec:intro}

For a given point set $P$ in the plane, in $d$-dimensional Euclidean space, or, even more generally,
in an arbitrary metric space, its \textit{\Term} is a directed graph based on minimum
distances. More precisely, these points are the vertices of the \Term, and each vertex has precisely
one outgoing edge towards its closest neighbor. To make this notion well-defined, we assume that our
point set is in \emph{general position}, which here means that no three points determine an
isosceles triangle. 

These graphs play an important role in modern Computational Geometry due to their relevance
in the context of computing geometric shortest paths~\cite{Mi00,MM17}.
Further applications of \Term s for computing spanners,
well-separated pairs, and approximate minimum spanning trees in $\RR^d$ can be found in the
survey~\cite{Sm00}, books~\cite{BCK+08,PS85}, or monograph~\cite{NS07}, see also~\cite{HJLM93}. A
systematic study of the basic combinatorial properties of \Term s dates back at least to the
classical paper~\cite{EPY97} by Eppstein, Paterson, and Yao in which, among many other results, they
made the following simple observation: two edges with the same endpoint meet at an angle of at least
$\pi/3$. Hence, the maximum indegree of the \Term\ of a point set $P \subset \RR^d$ is bounded from above
by the \textit{kissing number} of $\RR^d$, see~\cite{BDM+12}, \ie, the maximum number of non-overlapping unit spheres that can touch a unit sphere in $\RR^d$, regardless of the size of $P$. 

Here we extend the latter result and study the maximum indegree in a closely related class of
\textit{ordered \Term s} introduced in~\cite{AEM92,Epp92} in the context of dynamic algorithms. In
this case, the vertices appear one by one, and each new vertex has precisely one outgoing edge
towards its closest predecessor, see Figure~\ref{fig:one}. 

\begin{figure}[hbtp]
\centering
\includegraphics[scale=1.1]{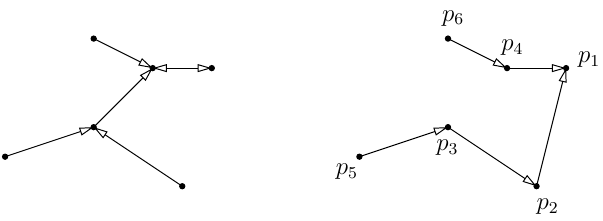}
\caption{Unordered (left) and ordered (right) \Term s on the same set of six points.}
\label{fig:one}
\end{figure}

It is not hard to see that every point set admits an order such that the corresponding ordered
\Term\ is a path, and thus its maximum indegree is $1$ (Proposition~\ref{thm:min_max-degree}). We
therefore focus on a `dual' problem, where the goal is to construct an order that \textit{maximizes}
the maximum indegree. We have the following results, partially addressing this question in three
different settings. 

For $n$ points on the line, the maximum indegree can be made $\lceil \log n \rceil$, which is optimal.
(Unless specified otherwise, all logarithms in this paper are in base~$2$.) 

\begin{theorem} \label{thm:line}
	For every set of $n$ points on the line, there exists an order such that the corresponding
        ordered \Term\ has maximum indegree at least $\lceil\log{n}\rceil$.  On the other hand,
        there is a set of $n$ points on the line such that for every order, the indegree of each
        point is at most $\lceil\log{n}\rceil$.  
\end{theorem}

We obtain a similar maximum indegree (with a loss factor $1/(4d)$) for $n$ points in $\RR^d$,
and in view of Theorem~\ref{thm:line}, this bound is best possible apart from the loss factor. 

\begin{theorem} \label{thm:d-space}
	For every set of $n$ points in $\RR^d$, there exists an order such that the corresponding
        ordered \Term\ has maximum indegree at least $\log{n}/(4d)$.  
\end{theorem}

The maximum indegree we manage to obtain in a metric space on $n$ points is only about 
$\sqrt{\log{n}}$, and closing the gap between this lower bound and the $\lceil\log{n}\rceil$ upper
bound remains as a challenging open problem. 

\begin{theorem} \label{thm:metric2}
For every $n$-element metric space, there exists an order such that the corresponding ordered
\Term\ has maximum indegree $\Omega\big(\sqrt{\log(n)/\log\log(n)}\big)$.  
\end{theorem}

\paragraph{Related work.}
For a point set in the plane, the notion of (ordered or unordered) \Term\ admits the following
generalization. The space around each point is divided into $k$ cones of equal angle similarly
positioned around each point, and each point is connected to a `nearest' neighbor in each
cone. Depending on the definition of `nearest' (either in a sense of the Euclidean distance or in
terms of the projection on the cones' bisectors), the resulting graph is called either a
\textit{Yao} or a \textit{Theta Graph}. These notions were introduced by Yao~\cite{Yao82} and
Clarkson~\cite{Cla87}, respectively, and their ordered variants are due to Bose, Gudmundsson, and
Morin~\cite{BGM04}. These graphs are useful in constructing spanners with nice additional
properties, such as logarithmic maximum degree and logarithmic diameter that are obtained for
suitable insertion orders, see also~\cite{Epp00}. Sparse graphs with a small dilation have been
studied in~\cite{ABC+08,BGM04,BGS05,BMN+04}, among others. 

Note that the special case $k=1$ retrieves the original definition of the (ordered) \Term. However,
for larger $k$, it is an open problem in the field to determine whether every point set $P$ admits
an order such that the maximum indegree of the corresponding ordered Yao or Theta Graph is bounded
from above by some constant $c_k$, independent of the size of $P$, see~\cite{BGM04}. To get a better
understanding of how degrees behave in these graphs, here we attempt to \textit{maximize} the
maximum indegree. The general case $k \geq 2$ will be addressed in a subsequent paper. 

In contrast, constructing an order \emph{minimizing} the maximum indegree of the ordered \Term\ is
rather straightforward. 

\begin{proposition} \label{thm:min_max-degree}
	For every point set $P$, there is an order such that the corresponding ordered \Term\ is a
        path. Moreover, any point in $ P$ can be chosen to be the tail of this path.  
\end{proposition}
\begin{proof}
Write $n = |P|$ and fix an arbitrary point $p_n \in P$. For $1\leq i\leq n-1$, let $p_{n-i}$ be the
point closest to $p_{n-i+1}$ among the points $P\setminus \{p_n,p_{n-1},\dots, p_{n-i+1}\}$.  
Now consider the order $\pi=(p_1,p_2,\dots,p_n)$. It is easy to check that the ordered \Term\ $G$
corresponding to $\pi$ is just the path
$p_n \rightarrow p_{n-1}  \rightarrow \ldots  \rightarrow p_1$. Indeed,
at the moment when $p_i$ is introduced, we connect $p_i$ to the point among  
 $\{p_1,p_2,\dots,p_{i-1}\}$ that is closest to $p_i$, which is exactly $p_{i-1}$ according to our
construction of $\pi$. 
\end{proof}

\paragraph{Notation.}
For a finite point set $P \subset \RR^d$, we use $\diam(P)$ to denote its diameter, \ie, the maximum
distance between any pair of points in $P$, and $\conv(P)$ to denote its convex hull.

\section{Points on the line: proof of Theorem~\ref{thm:line}} \label{sec:line}
	
\emph{Upper bound.} We construct the desired set of points on the line
recursively. For $n=2^1$, we put $P=P_1 = \{0,1\}$. For $k\ge 1$ and
$n=2^{k+1}$, we put $P=P_{k+1} = P_k \cup (3^k+P_k)$. 
It is straightforward to verify by induction that $\diam(P_k) = (3^k-1)/2$,
which is smaller than the distance between the two `halves' of $P_{k+1}$,
namely, $P_k$ and $3^k+P_k$.
Indeed, $\diam(P_k) = (3^k-1)/2 <  (3^k+1)/2$.
Therefore, in every order of $P_{k+1}$, each point can
have at most one edge incoming from the other `half'. Hence by induction, the
indegree of each point does not exceed $k+1 =\log{n}$. Finally, if
$2^k < n \le 2^{k+1}$, then an arbitrary subset $P \subseteq P_{k+1}$ of
size $n$ inherits the latter property.

\smallskip \noindent
\emph{Lower bound.} 
Choose $k$ such that $\lceil\log{n}\rceil = k+1$, \ie, $2^k<n \le 2^{k+1}$.
The proof is by induction on $k$.
Note that the statement is trivial for $k=0$, so let us now assume that $k \ge 1$.
We identify a suitable subset $P' \subseteq P$ such that its reveal already
forces the indegree of one of the points to reach the required threshold.
The remaining points in $P \setminus P'$ are introduced afterwards and clearly do not
decrease any degree. 
We may assume that $P$ lies on the $x$-axis. Let $X \subseteq P$.
We outline a recursive procedure $\Order$ for generating a suitable order
that employs the following subroutines and a distinguished element of $P$:

\begin{itemize} \itemsep 2pt
\item $\cent(X)$: refers to a specified element of $X$ that has a high indegree
  (to be determined). If $X=\{x\}$, then $\cent(X)=x$.
\item $\Order(X)$: lists the elements of $X$ in a suitable order that ensures 
  the existence of $\cent(X)$. Moreover,  $\cent(X)$ appears first in this order.
\item $\AnyOrder(X)$: lists the elements of $X$ in an arbitrary order;
 for instance from left to right.
\item $\Delete(L,x)$: deletes $x$ from list $L$ and returns the resulting list.
\end{itemize}

\newpage
%\medskip
\noindent {\bf Algorithm~$\Order(P)$}
\begin{enumerate} \itemsep 2pt

\item [] Step 1:
Let $a$ and $b$ denote the leftmost and rightmost points of $P$.

\item [] Step 2:
Partition $P= A \cup B$ around the midpoint of $ab$: let $A$
be the set of points of $P$
that are closer to $a$ than to $b$ (including $a$ itself), while $B$
is the set of remaining points;
assume without loss of generality that 
$|A| \geq |B|$; if $|A| < |B|$ proceed analogously by switching $A$ with $B$ and
$a$ with $b$. 

\item [] Step 3:
List $P$ as follows:
$\cent(A)$, $b$, $\Delete(\Order(A),\cent(A))$, \\$\AnyOrder(B \setminus\{b\})$.

\item [] Step 4:
$\cent(P) \gets \cent(A)$.

\end{enumerate}

\paragraph{Analysis.} Let $\Delta^{-}(P)$ denote the indegree of $\cent(P)$.
Note that $\Order(A)$ is a list obtained by a recursive call and that the distance 
from $b$ to $A$ is strictly larger than the diameter of $A$ by construction.
Therefore, the indegree of $\cent(A)$ is increased by $1$
due to the directed edge $b \to \cent(A)$ that is added when $b$ is revealed to the
algorithm as the second point. As a result,
\[ \Delta^{-}(P) \geq \Delta^{-}(A) +1 \ge k+1,\]
where the last inequality holds by induction hypothesis since $2^{k-1} < |A|$. 
\qed

\section{Points in $\RR^d$: proof of Theorem~\ref{thm:d-space}} \label{sec:d}

Here we further develop the ideas from the previous subsection to translate them from the line 
to $\RR^d$, for any fixed $d$. As the main supplementary tool, we use the following standard result
in discrete geometry, that can be deduced, \eg, from \cite[Ineq. (4) and (6)]{RZ97}, see also~\cite{BR25} 

\begin{theorem} \label{thm:RZ}
Let $K$ be a convex set in $\RR^d$, $K_0$ be its interior, and $r<1$ be a positive number.
Then $K$ can be covered by at most $(1+1/r)^d(d\ln d+d\ln \ln d + 5d)$ translates of $-rK_0$,
reflected scaled copies of $K_0$. 
\end{theorem}

\begin{corollary} \label{cor:borsuk}
  Let $P$ be a finite set of points in $\RR^d$ such that $\diam(P) \leq 1$.
  Then $P$ can be partitioned into at most $16^d/2$ subsets of diameter less than $1/2$.
\end{corollary}

\begin{proof}
    We apply Theorem~\ref{thm:RZ} with $K=\conv(P)$ and $r=1/2$. Note that $\diam(K)= \diam(P) \le 1$,
    and thus the distance between any two points of $-(1/2)K_0$ is \textit{strictly less} than $1/2$.
    It remains only to note that the number of translates given by Theorem~\ref{thm:RZ} is smaller
    than $16^d/2$ for all $d \ge 2$. 
\end{proof}

Now we construct the desired order of $P$ with large maximum indegree of the corresponding ordered
\Term\ by the following algorithm, which is a higher-dimensional extension of the procedure we used
in Section~\ref{sec:line}.  

\medskip
\noindent {\bf Algorithm~$\Order(P)$}

\begin{enumerate} \itemsep 2pt

\item[] Step 1: Compute a diameter pair $ab$, where we may assume that $|ab|=1$.

\item[] Step 2: Let $A=\{ p \in P \colon |pa| \leq |pb| \}$, and
$B=\{ p \in P \colon |pb| \leq |pa| \}$. Assume w.l.o.g. that $|A| \geq |B|$, thus $|A| \geq n/2$.

\item[] Step 3: By Corollary~\ref{cor:borsuk}, partition $A$ into at most $16^d/2$ subsets where
  each subset has diameter less than $1/2$. One of these subsets  $C \subseteq A$ contains at least
  $n/16^d$ points.  

\item[] Step 4: List $P$ in the following order: 
\[ \cent(C), b, \Delete(\Order(C),\cent(C)), \AnyOrder\big(P \setminus (C \cup\{b\})\big). \]

\item[] Step 5: $\cent(P) \gets \cent(C)$

\end{enumerate}

\paragraph{Analysis.}
$\Order(P)$ is a list obtained by a recursive call. Observe that in $\Order(P)$, the element
$\cent(P)$ is listed first, as required. As earlier, let $\Delta^{-}(P)$ denote its indegree. By
construction, the indegree of $\cent(C)$ is increased by $1$ due to the directed edge $b \to
\cent(C)$ that is added when $b$ is revealed to the algorithm as the second point. Moreover, for
each subsequent step that reveals an element $c \in C$, since $\diam(C) <1/2$ and $|cb| \geq 1/2$,
$C$ is processed in a manner independent of the presence of $b$. As a result, 
\[ \Delta^{-}(P) \geq \Delta^{-}(C) +1. \]

We now show that $\Delta^{-}(P) \geq \log_{16^d}{n} =\log{n}/(4d)$.
This inequality is clearly satisfied for $n=1,2$.
By induction, it is verified that 
\[ \Delta^{-}(P) \geq \Delta^{-}(C) +1 \geq \log_{16^d}{(n/16^d)} + 1 = \log{n}/(4d) -1 +1 =
\log{n}/(4d). ~~~~~~ \qed \]

\section{Abstract metric spaces: proof of Theorem~\ref{thm:metric2}} \label{sec:metric}

Let $\rho$ be a metric on a finite set $V= \{v_1,\dots,v_n\}$, namely a positive and symmetric
function $\rho: V\times V \to \RR_+$ satisfying the triangle inequality. So far we labeled the
points of $V$ arbitrarily just for definiteness, the desired order on them will be constructed
later. We assume familiarity with hypergraphs, otherwise readers are referred to~\cite{HF21}
for related definitions.

We define a $3$-coloring of a complete $3$-uniform hypergraph $K_n^{(3)}$ on the vertex set $[n]$ as
follows: for all $1\le i_1 < i_2 < i_3 \le n$, we 
        color the triple $\{i_1,i_2,i_3\}$ by red (resp. green or
        blue) whenever $v_{i_2}v_{i_3}$ (resp. $v_{i_1}v_{i_3}$ or
        $v_{i_1}v_{i_2}$) is the `shortest side' of a triangle
        $v_{i_1}v_{i_2}v_{i_3}$, that is $\rho(v_{i_2},v_{i_3})< \rho(v_{i_1},v_{i_2})$
        and $\rho(v_{i_2},v_{i_3})< \rho(v_{i_1},v_{i_3})$.
        Recall that our points are in general position, meaning that no
        point triple determines an isosceles triangle, and so each triple gets exactly one color. A
        subhypergraph of $K_n^{(3)}$ is \emph{monochromatic} if all its triples are of the same
        color. A \textit{forward star} $S_k^{(3)}$ is a 3-uniform hypergraph on an ordered vertex
        set $\{i_1, \ldots , i_k\} \subseteq [n]$, labeled in the ascending order, consisting of edges
        $\{i_1,i_j,i_{j'}\}$ for all 
    $i_{j}<i_{j'}$. We claim that if there exists a red clique $K_k^{(3)}$, or a green
    forward star $S_k^{(3)}$, or a blue forward star $S_k^{(3)}$, then there exists
    an order such that the corresponding \Term\ has maximum indegree at least $k-1$. 
    To verify this claim, suppose that such a structure exists on the vertex set 
    $I=\{i_1, \ldots , i_k\}$, labeled in the ascending order, and consider the following
    three cases. 
	
	First, suppose that these vertices form a red clique. We claim that under the order
	\[v_{i_k},v_{i_1},v_{i_2},\dots,v_{i_{k-1}}, \AnyOrder(V\setminus I),\]
	the indegree of $v_{i_k}$ is at least $k-1$. Indeed, note that
        as each new vertex $v_{i_j}$ is revealed,
        $\rho(v_{i_j},v_{i_k})$ is the shortest distance among all
        currently visible points. Hence, $v_{i_j}\to v_{i_k}$ is an
        edge of the ordered \Term\ for all $1 \le j \le k-1$. 
	
	Similarly, if these vertices form a blue forward star, then the indegree of $v_{i_1}$ under the order
	\[v_{i_1},v_{i_k},v_{i_{k-1}},\dots,v_{i_2}, \AnyOrder(V\setminus I),\]
	is at least $k-1$. Indeed, when $v_{i_j}$ is revealed, the
        distance $\rho(v_{i_1},v_{i_j})$ will be the shortest among
        all current points as $\{i_1,i_j,i_\ell\}$ is always blue. 
	
	Finally, if these vertices form a green forward star, then the indegree of $v_{i_1}$ under the order
	\[v_{i_1},v_{i_2},v_{i_3},\dots,v_{i_k}, \AnyOrder(V\setminus I),\]
	is at least $k-1$. Indeed, when $v_{i_j}$ is revealed, an
        arbitrary currently visible point $v_{i_\ell}$ satisfies
        $\{i_1,i_\ell,i_j\}$ being green. So the distance
        $\rho(v_{i_1},v_{i_j})$ will be the shortest. 
	
	The last ingredient of the proof is the following Ramsey-type result due to He and Fox; more
        precisely, their Theorem 1.2 and Lemma 5.2 in ~\cite{HF21}. 
    We rewrite this proof to make it explicit that the argument in \cite{HF21} produces forward
    stars and also extends to more than two colors. 
  
    \begin{theorem}\label{thm:ramsey}
    Let $K_n^{(3)}$ be a complete 3-uniform hypergraph on an ordered
    vertex set $[n]$, and its edges be colored by either red,
    green, or blue. If $K_n^{(3)}$ contains neither a red clique
    $K_k^{(3)}$, nor a green forward star $S_k^{(3)}$, nor a blue forward star
    $S_k^{(3)}$, then $n \leq \exp\left(O(k^2\log(k))\right)$. 
    \end{theorem}

	Applying this theorem with $k=c'\sqrt{\log(n)/\log\log(n)}$
        for a sufficiently small $c'>0$, we ensure the existence of a
        monochromatic special structure, and thus the existence of an
        order such that the corresponding ordered \Term\ has maximum
        indegree at least $k-1$, which completes the proof of Theorem~\ref{thm:metric2}. 

\begin{proof}[Proof of Theorem~\ref{thm:ramsey}]
  We consider a process of picking and deleting vertices from $V(K_n^{(3)})$,
  and constructing an auxiliary graph $G$.
  We keep track of two sets $U$, the vertex set of $G$, and $W$, the set of
    vertices waiting to be picked. At the beginning of the process, we have
    $U=\emptyset$ and $W=V(K_n^{(3)}) = [n]$. In each iteration, we pick the smallest
    vertex $v \in W$, and delete it from $W$. Then, following a rule
    we shall describe, we create edges from $v$ to vertices in $U$,
    assign $v$ a color (red/green/blue), and add it into $U$. During
    this step, immediately after we created an edge $\{u,v\}$, we look
    at every vertex $w\in W$: if at least $|W|/k$ vertices $w$ satisfy that
    $\{u,v,w\}$ is green, then we color the edge $\{u,v\}$ green
    (and vertex $v$ green, see further below),
    and delete all $w \in W$ with $\{u,v,w\}$ not green; otherwise, 
    we perform the same checking, coloring, and deleting with 
    ``blue'' in place of ``green''; if neither of these cases happens, 
    we color the edge $\{u,v\}$ by red, and delete all $w\in W$ such that 
    $\{u,v,w\}$ is not red.

    Now we describe how we create edges: once a new vertex $v$ is picked, we keep creating
    edges from $v$ to the red vertices in $U$ following the order
    until we encounter a green or blue edge; if the last edge we
    created is green (resp. blue), we color the vertex green
    (resp. blue); otherwise we color this vertex red. The new vertex
    and the created edges are considered as added into $G$.
    By construction, every edge in $G$ leads to a red vertex
    and each green (resp. blue) vertex is adjacent to exactly one green (resp. blue) edge.
    The next iteration begins as long as there are still vertices left in $W$.
  
    If there are $k$ red vertices in $G$, then by construction all the 2-edges between them are red.
    And again by the way we color the edges in $G$, these vertices correspond to a
    3-uniform red clique $K_k^{(3)}$ inside $K_n^{(3)}$. Hence, without loss of
    generality, we may assume that there are fewer than $k$ red vertices in $G$.
    Given this assumption, if there are $(k-1)^2$ green vertices
    in $G$, then by our construction and the pigeonhole principle,
    there must be one red vertex $v_1$ and $k-1$ green vertices
    $v_2,\dots,v_{k}$ such that $\{v_1,v_i\}$ is green for $2\leq
    i\leq k$. Then by the way we color the edges in $G$, these vertices
    correspond to a 3-uniform green forward star $S_k^{(3)}$ inside
    $K_n^{(3)}$. Hence, without loss of generality, there are fewer than
    $(k-1)^2$ green vertices in $G$, and a similar claim holds for
    blue vertices as well. Therefore, $G$ has fewer than $k+2(k-1)^2$
    vertices in total. 

    Next, we upper bound the number of edges in $G$. Since green edges only originate from green vertices,
    there are fewer than $(k-1)^2$ green edges. Similarly, there are fewer than $(k-1)^2$ blue edges. There
    are at most $(k-1)^2$ red edges between red vertices. To count
    other red edges, notice that the $i$-th red vertex has fewer than
    $k-1$ green (resp. blue) edges adjacent to it, and each endpoint
    of these edges is adjacent to $i-1$ red edges. In summary, the
    number of red edges in $G$ is fewer than
    
    \begin{equation*} 
        (k-1)^2 + \sum_{i=1}^{k-1}(i-1)(k-1) = k(k-1)^2/2.
    \end{equation*}

    Finally, notice that each vertex in $G$ decreases the size of $W$
    by 1, each green or blue edge preserves at least a $1/k$-fraction of the elements of $W$
    and each red edge preserves at least a $(1-2/k)$-fraction of the elements of $W$.
    Since by the end of our construction of $G$ when $W$ is depleted, $G$ has fewer
    than $k+2(k-1)^2$ vertices, fewer than $2(k-1)^2$ green or blue
    edges, and fewer than $k(k-1)^2/2$ red edges, the size of $W$ in
    the beginning, namely $n$, is at most

    \begin{equation*} 
        (k+2(k-1)^2) (1/k)^{-2(k-1)^2} (1-2/k)^{-k(k-1)^2/2} \leq \exp\left(O(k^2\log(k))\right).
    \end{equation*}
    Here, we used an elementary estimate that $(1-2/k)^{-k/2} < e^2$. 
\end{proof}

\section{Concluding remarks} \label{sec:remarks}

Recall from the second half of Theorem~\ref{thm:line} that  there exists a set of $n$ points on the
line such that for every order, the indegree of each point in the corresponding \Term\ is at most
$\lceil\log{n}\rceil$. Unlike the linear case, we do not know if this construction is close to being
tight for larger dimensions or abstract metric spaces.

Note that both in the definition of \Term\ corresponding to a metric space and in our proof of
Theorem~\ref{thm:metric2}, we didn't really use the triangle inequality, and thus we could argue in
a more general setting of \textit{semimetric spaces}. Moreover, the \textit{exact} distances between
the points are also not important, since only their \textit{relative order} determines the
\Term. Furthermore,  we know from~\cite{AM22} 
that any order of the distances on an $n$-element set can be realized in the $(n-1)$-dimensional
Euclidean space. Hence, a sufficiently strong improvement upon the result of
Theorem~\ref{thm:d-space} could strengthen the lower bound from Theorem~\ref{thm:metric2} as well. 

Observe that our proof of Theorem~\ref{thm:metric2} works without any changes if instead of a red
clique, we could guarantee only a red `backward star', which is a much sparser hypergraph.  
(A \textit{backward star} is a 3-uniform hypergraph similar to a forward star, where the only
difference is that the common vertex of all the hyperedges has not the smallest, but the largest
index.) We are unaware of any variant of Theorem~\ref{thm:ramsey} that guarantees a backward star  
in one color or a forward star in (one of) the other color(s) for a much smaller $n$. Note that the
proof of Theorem~\ref{thm:ramsey} works for \textit{any} coloring of the complete hypergraph, while
colorings that come from metric spaces satisfy additional constraints due to
transitivity\footnote{For instance, in the notation of Section~4, it is not possible that two
  triples $\{1,2,3\}$ and $\{1,3,4\}$ are blue, while $\{1,2,4\}$ is green, because otherwise we
  would have $v_1v_2 < v_1v_3 < v_1v_4 < v_1v_2$, a contradiction.}. 

A closely related problem is the following:

\begin{problem} \label{Prob1}
	For an $n$-element metric space $V$ and $v \in V$, let $d(v)$ be the maximum indegree of $v$
        in the ordered \Term\ over all $n!$ possible orders. Can $\sum_v 2^{-d(v)}$ be larger
        than~$1$? 
\end{problem}

On the one hand, note that if the sum $\sum_v 2^{-d(v)}$ is at most one, then $d(v)\ge
\lceil\log{n}\rceil$ for at least one $v\in V$, which would strongly improve upon the result of
Theorem~\ref{thm:metric2}. On the other hand, if $\sum_v 2^{-d(v)} > 1$ for some $V$, then all the
points of an appropriate `iterated blow-up' $Q$ of $V$ satisfy $d(q) < c\log {|Q|}$ for some $c<1$
and all $q \in Q$, and this would improve the aforementioned upper bound from the second half of
Theorem~\ref{thm:line}. 

We managed to verify the inequality $\sum_v 2^{-d(v)} \le 1$ for all $n \le 5$ by computer
search. Moreover, it is not hard to check that this sum equals exactly $1$ for the metric spaces
constructed in the second half of the proof of Theorem~\ref{thm:line}, and that these are not the
only examples. However, computer experiments suggest that such metric spaces become rare as $n$
grows. This can be tentatively considered as evidence towards a negative resolution of
Problem~\ref{Prob1}.

\paragraph{Acknowledgments.}
The authors thank the organizers of the Focused Week on Geometric Spanners (Oct 23 -- Oct 29, 2023)
at the Erd\H os Center, Budapest, where this joint work began. Research partially
supported by ERC Advanced Grant `GeoScape' No. 882971 and by the Erd\H os Center and by the Ministry of Innovation and Technology NRDI Office within the framework of the Artificial Intelligence National Laboratory (RRF-2.3.1-21-2022-00004). Research was also
partially supported by the Overseas Research Fellowship of IIIT-Delhi, India.

\end{document}